\documentclass[a4paper, 12pt]{article}

\usepackage[english]{babel}
\usepackage[T1]{fontenc}
\usepackage[utf8]{inputenc}
\usepackage{amsthm}
\usepackage{amsmath}
\usepackage{graphicx}
\usepackage{amssymb}
\usepackage{latexsym}
\usepackage{color}
\usepackage{url}
\usepackage{subcaption}

\theoremstyle{definition}
\newtheorem{definition}{Definition}
\newtheorem{theorem}[definition]{Theorem}
\newtheorem{corollary}[definition]{Corollary}
\newtheorem{lemma}[definition]{Lemma}

\newtheorem{proposition}[definition]{Proposition}
\theoremstyle{definition}
\newtheorem*{problem*}{Problem}
\newcommand{\dist}{\operatorname{dist}}

\begin{document}

\title{Decomposition of cubic graphs\\with cyclic connectivity 5}

\author{
	Edita Máčajová, Jozef Rajník
	\\[3mm]
	\\{\tt \{macajova, rajnik\}@dcs.fmph.uniba.sk}
	\\[5mm]
	Comenius University, Mlynská dolina, 842 48 Bratislava\\
}

\maketitle

\begin{abstract}
Let $G$ be a cyclically $5$-connected cubic graph with a $5$-edge-cut separating $G$ into two cyclic components $G_1$ and $G_2$. We prove that each component $G_i$ can be completed to a cyclically $5$-connected cubic graph by adding three vertices, unless $G_i$ is a cycle of length five.
Our work extends similar results by Andersen et al. for cyclic connectivity $4$ from 1988.

\end{abstract}

\section{Introduciton}

The study of cubic graphs offers a convenient approach to several widely-open conjectures such as the Tutte's $5$-flow conjecture, the cycle double cover conjecture, or the Berge-Fulkerson conjecture. It is known that minimal counterexamples to aforementioned conjectures are cubic graphs. Moreover, various other requirements that a minimal counterexample has to satisfy have been studied. Perhaps the most notorious requirement is that any cubic counterexample should not be $3$-edge-colourable, or in other words, that it is a \emph{snark}.

In the study of minimal counterexamples, connectivity plays a~crucial role. Since the connectivity of a cubic graph cannot exceed three, it is advantageous to use a refined measure---the \emph{cyclic connectivity}, that is, the minimum number of edges needed to separate two cycles in a given cubic graph.
It has been proven that the minimal counterexample to the $5$-flow conjecture is cyclically $6$-connected \cite{Kochol04-5f-cc6}, cyclically $4$-connected for the cycle double cover conjecture \cite{Zhang} and cyclically $5$-connected for the Berge conjecture \cite{Macajova20-bf-cc5}. On the other hand, Jaeger and Swart conjectured that there are no cyclically $7$-connected snarks \cite{Jaeger}.

Cubic graphs with small edge cuts enable us to use inductive arguments. If $G$ is a graph from some class $\mathcal{C}$ with a small cycle-separating cut, it is useful, if possible, to decompose $G$ along the cut into two smaller graphs contained in $\mathcal{C}$. Andersen et al. \cite{Andersen88} established such results for the class of cyclically $4$-connected cubic graphs. They showed that each \emph{cyclic part}, that is, a component separated by a cycle-separating cut of minimal size, of a cyclically $4$-connected cubic graph can be extended to a cyclically $4$-connected cubic graph by adding a pair of adjacent vertices and restoring $3$-regularity. Moreover, they characterised graphs where it is sufficient to add only two additional edges. Using this result they proved a lower bound on the number of removable edges in a cyclically $4$-connected cubic graph \cite{Andersen88}. Later, Goedgebeur et al. constructed and classified all snarks with cyclic connectivity $4$ and oddness $4$ up to order $44$ \cite{Goedgebeur, Goedgebeur-oddness-all}.

In this paper, we examine how a cyclic part $H$ of a cubic graph with cyclic connectivity $5$ can be completed to a cyclically $5$-connected cubic graph. We show that, except for the case where $H$ is a cycle of length 5, it is sufficient to add to $H$ three vertices on a path of length two and restore $3$-regularity to obtain a cyclically $5$-connected cubic graph. In Section \ref{sec:prelm} we summarise notions and a result concerning cyclic connectivity that we shall use. Then, in Section \ref{sec:lemmas}, we establish a weaker result that we can complete $H$ to a cubic graph with girth at least $5$. We prove our main result in Section \ref{sec:main} by showing that if we complete $H$ to a cubic graph with girth at lest $5$ which is not cyclically $5$-connected, we can use the structure of $H$ to find another completion which yields a cyclically $5$-connected cubic graph.

\section{Preliminaries}
\label{sec:prelm}

We start with some basic definitions and notations. All considered graphs are cubic and may contain loops and parallel edges, although their existence will be often excluded by additional requirements. We denote the subgraph of a graph $G$ induced by a set of vertices $X$ by $G[X]$. The set of edges of the graph $G$ that have one end in $X$ and the other in $V(G) - X$ is denoted by $\delta_G(G[X])$, or $\delta_G(X)$. We omit the subscript $G$ whenever the graph is clear from the context. Also, we will write only $\deg_X(v)$ instead of $\deg_{G[X]}(v)$ to denote the degree of vertex $v$ in the induced subgraph $G[X]$.

An \emph{edge-cut} of a connected graph $G$, or a \emph{cut} for short, is any set $S$ of edges of $G$ such that $G - S$ is disconnected. An edge-cut is \emph{cycle-separating} if at least two components of $G - S$ contain a cycle. We say that a connected graph $G$ is \emph{cyclically $k$-edge-connected} if it contains no cycle-separating edge-cut consisting of fewer than $k$ edges. The \emph{cyclic edge-connectivity} of $G$, denoted by $\zeta(G)$, is the largest number $k \le \beta(G)$, where $\beta(G) = |E(G)| - |V(G)| + 1$ is the cycle rank of $G$, for which $G$ is cyclically $k$-edge-connected (cf. \cite{Nedela-Atoms}, \cite{Robertson}).

The cyclic edge-connectivity is bounded by the cycle rank, because there are graphs, where any two cycles share an edge. Such graphs are cyclically $k$-connected for every positive integer $k$. Among the simple cubic graphs, the only such examples are $K_4$ and $K_{3,3}$ for which we have $\zeta(K_4) = 3$ and $\zeta(K_{3,3}) = 4$. The cyclic connectivity of every graph $G$ is bounded from above by the \emph{girth} of the graph $G$, denoted by $g(G)$, which is the length of a shortest cycle in $G$ \cite{Nedela-Atoms, Robertson}.

One can easily check that for a cubic graph $G$ with $\zeta(G) \le 3$, the value $\zeta(G)$ is equal to the usual vertex-connectivity and edge-connectivity of $G$. Furthermore, cyclic edge-connectivity and cyclic vertex-connectivity, which is defined in a similar manner, of every cubic graph coincide. Therefore, we shall only use terms \emph{cyclically $k$-connected} and \emph{cyclic connectivity} instead of cyclically $k$-edge-connected and cyclic edge-connectivity.

Let us consider a cycle-separating edge-cut $S$ of minimum size. One can clearly see that $S$ consists of independent edges. Moreover, $G - S$ has exactly two components called \emph{cyclic parts} or \emph{fragments}. The following proposition \cite[Proposition 4]{Nedela-Atoms} of Nedela and Škoviera says that each cyclic part of a cyclically $5$-connected graph is $2$-connected.

\begin{proposition}
\label{prop:fragments}
Let $G$ be a connected cubic graph. Then each cyclic part of $G$ is connected. Moreover, if $\zeta(G)$ > 3, then each cyclic part is $2$-connected.
\end{proposition}

If $H$ is a non-empty induced subgraph of a cyclically $5$-connected cubic graph $G$, then it is either cyclic, and thus $|\delta_G(H)| \ge 5$, or $H$ is acyclic. In the latter case the relation between the number $|\delta_G(H)|$ and the number of vertices of $H$ is determined by following lemma, which can be proven by induction. Since $H$ is non-empty, we get bound on $|\delta_G(H)|$.

\begin{lemma}
\label{lemma:acyclic-pole}
Let $M$ be a connected acyclic induced subgraph of a cubic graph $G$. Then $|\delta_G(M)| = |V(M)| + 2$.
\end{lemma}

\begin{corollary}
\label{coro:connected}
If $M$ is a non-empty induced subgraph of a cyclically $5$-connected cubic graph $G$, then $|\delta_G(M)| \ge 3$.
\end{corollary}

In general, a cyclic induced subgraph $H$ of a cyclically $5$-connected cubic graph with $|\delta_G(H)| = 6$ need not be $2$-connected, since $H$ may contain a bridge. However, as we show in the following lemma, $H$ contains only one bridge which is additionally in a special position.

\begin{lemma}
\label{lemma:6pole}
Let $H$ be a connected induced subgraph of a cyclically $5$-connected cubic graph $G$ such that $|\delta_G(H)| = 6$. Then exactly one of the following holds:
\begin{itemize}
\item[(i)] $H$ is acyclic;
\item[(ii)] $H$ contains exactly one bridge whose one end $x$ is incident with two edges from $\delta_G(H)$ and $H - x$ is a $2$-connected cyclic part of $G$;
\item[(iii)] all the edges from $\delta_G(H)$ are independent and $H$ is $2$-connected.
\end{itemize}
\end{lemma}

\begin{proof}
If $H$ is acyclic then only (i) holds and if $H$ is $2$-connected then only (iii) is true because if some of the edges from $\delta_G(H)$ were adjacent, there would be a bridge in $H$. So it is sufficient to show that if $H$ is cyclic and contains a bridge, then (ii) holds true.

Suppose that $H$ contains a bridge which separates $H$ into components $C_1$ and $C_2$. Since $H$ is cyclic, at least one of the $C_1$ and $C_2$ has to contain a~cycle, say $C_1$.
Therefore $|\delta_G(C_1)| \ge 5$, so $|\delta_G(C_1) \cap \delta_G(H)| \ge 4$. By Corollary \ref{coro:connected} we have that $|\delta_G(C_2)| \ge 3$, so $|\delta_G(C_2) \cap \delta_G(H)| \ge 2$. Since $|\delta_G(H)| = 6$, we get that $|\delta_G(C_1) \cap \delta_G(H)| = 4$ and $|\delta_G(C_2) \cap \delta_G(H)| = 2$.
Thus $C_2$ contains only one vertex and that vertex is incident with two edges from $\delta_G(H)$. Moreover, since $C_1$ is cyclic and has five outgoing edges, $C_1$ is a~fragment and hence cyclically $2$-connected due to Proposition \ref{prop:fragments}. So (ii) is satisfied, which concludes our proof.
\end{proof}

Finally, we formalise the process of completing a cyclic part to a cubic graph by adding three new vertices lying on a path of length two.

\begin{definition}
Let $H$ be a cyclic part of a cubic graph $G$ with $\zeta(G) = 5$ and let $a_1$, $a_2$, $a_3$, $a_4$, and $a_5$ be the vertices of $H$ of degree $2$. We add to $H$ three vertices $x$, $y$ and $z$ and edges $xy$, $yz$, $xa_1$, $xa_2$, $ya_3$, $za_4$, and $za_5$. We denote the graph obtained in this way by $H(a_1, a_2, a_3, a_4, a_5)$. Throughout this paper, the three newly added vertices will be consistently denoted by $x$, $y$ and $z$.
\end{definition}

\section{Extensions without short cycles}
\label{sec:lemmas}

In this section we show that each cyclic part $H \ncong C_5$ of a cubic graph with $\zeta(G) = 5$ can be extended to a cubic graph $\bar{H} = H(a_1, a_2, a_3, a_4, a_5)$ which has girth at least $5$.

\begin{lemma}
\label{lemma:degrees}
Let $H$ be a cyclic part of a cubic graph with $\zeta(G) = 5$ which is not a $5$-cycle and let $A$ be the set of vertices of $H$ of degree $2$. Then each vertex from $A$ has at most one neighbour in $A$.
\end{lemma}

\begin{proof}
Let $a_2$ be a vertex from $A$ with two neighbours $a_1$ and $a_3$ in $A$. Then the induced subgraph $G[V(H) - \{a_1, a_2, a_3\}]$ has only four outgoing edges, hence it is acyclic and contains only two vertices due to Lemma \ref{lemma:acyclic-pole} which means that $H$ is a $5$-cycle (see Figure \ref{fig:g5-5cycle}); a contradiction.
\end{proof}

\begin{figure}[h]
\centering
\includegraphics{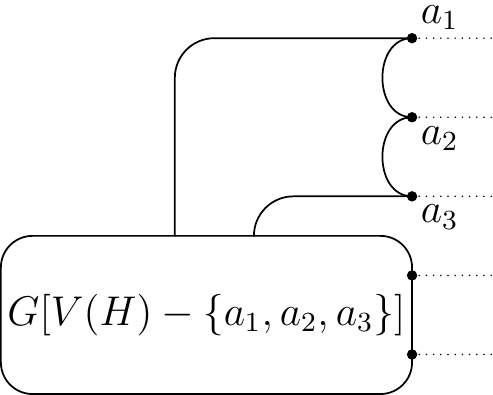}
\caption{The cyclic part $H$ in the case $a_2a_1, a_2a_3 \in E(H)$, the edges from $\delta_G(H)$ are dotted}
\label{fig:g5-5cycle}
\end{figure}

\begin{lemma}
\label{lemma:girth-extension}
Let $H$ be a cyclic part of a cubic graph $G$ with $\zeta(G) = 5$ that is not a $5$-cycle. Then there exists a permutation $a_1a_2a_3a_4a_5$ of vertices of $H$ of degree $2$ such that $H(a_1, a_2, a_3, a_4, a_5)$ has girth at least five.
\end{lemma}

\begin{proof}
Let $A = \{a_1, a_2, a_3, a_4, a_5\}$ be the set of the five degree $2$ vertices of $H$. Furthermore, let $D$ be the graph with vertex set $A$ where $a_ia_j \in E(D)$ if $\dist_H(a_i, a_j) = 2$ for each $a_i, a_j \in A$. Note that if $a_ia_j \in E(D)$, then there exist exactly one path $a_iva_j$ in $H$. Observe that the graph $H(a_1, a_2, a_3, a_4, a_5)$ has girth at least five if and only if
$$\deg_A(a_3) = 0 \qquad \text{and} \qquad a_1a_2, a_4a_5 \notin E(D).$$
By Lemma \ref{lemma:degrees}, there are at most two edges between the vertices from $A$. We divide the proof into three cases according to the number of edges in $H[A]$.

\paragraph{Case (i)} Assume that the induced subgraph $H[A]$ contains two edges, say $a_1a_5$ and $a_2a_4$.
We show that all the vertices $a_1$, $a_2$, $a_4$, and $a_5$ have at most one neighbour in $D$ among the vertices $\{a_1, a_2, a_4, a_5\}$. Suppose to the contrary that, say $a_1$ has two neighbours in $\{a_1, a_2, a_4, a_5\}$. Since obviously $a_1a_5 \notin E(D)$, we have $a_1a_2 \in E(D)$ and $a_1a_4 \in E(D)$, so there exist two paths $a_1v_2a_2$ and $a_1v_4a_4$ in $H$. However, since $\deg_H(a_1) = 2$ and $a_1$ is already adjacent to $a_5$, the vertices $v_2$ and $v_4$ coincide. It follows that the subgraph $H$ contains a $3$-cycle $a_2v_2a_4$---a contradiction.
Therefore, there exists a permutation $b_2b_4$ of $\{a_2, a_4\}$ such that $a_1b_2, a_5b_4 \notin E(D)$ and thus $g(H(a_1, b_2, a_3, b_4, a_5)) \ge 5$.

\paragraph{Case (ii)} Let $H[A]$ contain only one edge and denote it by $a_1a_5$ in such a way that $\deg_D(a_5) \le \deg_D(a_1)$. We show that $\deg_D(a_1) \le 2$. Suppose to the contrary that $\deg_D(a_1) \ge 3$. Then, there are three paths $a_1va_2$, $a_1va_3$ and $a_1va_4$ in $H$, each going through the same neighbour $v$ of $a_1$ because $a_1$ is already incident with $a_5$. Therefore, $\deg_H(v) = 4$, which is a contradiction. Therefore, one of the vertices $a_2$, $a_3$ and $a_4$ is not adjacent to $a_1$ in $D$, say $a_1a_2 \notin E(D)$. We show that $a_5a_4 \notin E(D)$ or $a_5a_3 \notin E(D)$. Suppose to the contrary that both the edges $a_5a_4$ and $a_5a_3$ are in $E(D)$. Since $2 = \deg_D(a_5) \le \deg_D(a_1) \le 2$, we have $a_1a_3, a_1a_4 \in E(D)$. The latter means, there are paths $a_1ua_3$, $a_1ua_4$, $a_5va_4$, and $a_5va_3$ in $H$.
However, then $a_3ua_4v$ is a $4$-cycle in $H$ which is a contradiction. Thus, one of the edges $a_5a_4$ and $a_5a_3$, say it is $a_5a_4$, is not in $E(D)$ and then the graph $H(a_1, a_2, a_3, a_4, a_5)$ has girth at least five.

\paragraph{Case (iii)} Finally, assume that $H[A]$ contains no edges. We show that one can choose four distinct vertices $b_1, b_2, b_4, b_5 \in V(D)$ such that $b_1b_2, b_4b_5 \notin E(D)$. It is a simple matter to verify that if this is not possible, then the graph $D$ contains $K_4$ as a subgraph or it contains two vertices of degree $4$.

At first, suppose that $D$ contains a $K_4$-subgraph consisting of the vertices $a_1$, $a_2$, $a_3$, and $a_4$. Then, there are paths $a_1v_2a_2$, $a_1v_3a_3$ and $a_1v_4a_4$ in $H$. However, since $\deg_H(a_1) = 2$ some of the vertices $v_2$, $v_3$, $v_4$ have to coincide, say $v_2 \equiv v_3$.
Analogously, there are also paths $a_4u_2a_2$ and $a_4u_3v_3$ in $D$ and some of the vertices $u_2, u_3, v_4$ have to coincide. However, then one of $(u_2 = u_3, a_3, v_3 = v_2, a_2)$, $(u_3 = v_4, a_4, v_3, a_3)$ or $(v_4 = u_2, a_2, v_2, a_1)$ is a $4$-cycle in $H$ and this is a contradiction.

Now suppose that $D$ contains two vertices of degree $4$, say $a_1$ and $a_5$. Then, we have paths $a_1v_2a_2$, $a_1v_3a_3$, $a_1v_4a_4$, and $a_1v_5a_5$ in $H$. Since $\deg_H(a_1) = 2$, we have $|\{v_2, v_3, v_4, v_5\}| = 4$, so say $v_2 = v_3$ and $v_4 = v_5$. Analogously, there are paths $a_5v_5a_1$, $a_5v_5a_4$, $a_5uv_3$, and $a_5uv_2$ in $H$. However $a_2v_2a_3u$ is a $4$-cycle in $H$ and that is a contradiction.
Therefore $g(H(b_1, b_2, b_3, b_4, b_5)) \ge 5$, where $b_3 \in A - \{b_1, b_2, b_4, b_5\}$.
\end{proof}

\begin{lemma}
\label{lemma:distribution}
Let $H$ be a cyclic part of a cubic graph $G$ with $\zeta(G) = 5$ and let $a_1$, $a_2$, $a_3$, $a_4$, and $a_5$ be the vertices of degree $2$ in $H$. Suppose that the graph $\bar{H} = H(a_1, a_2, a_3, a_4, a_5)$ has girth $5$ and that $\bar{H}$ contains a minimum cycle-separating cut $S$ of size smaller than $5$. Then $|S| = 4$ and the cut $S$ separates $\{a_1, a_2, x\}$ from $\{a_4, a_5, z\}$.
\end{lemma}

\begin{proof}
We start by showing that the vertices $x$ and $z$ are in different components $C_1$ and $C_2$ of $\bar{H} - S$. Suppose to the contrary that $x, z \in C_1$. Then the common neighbour $y$ of $x$ and $z$ is also in $C_1$, otherwise the minimal cut $S$ would contain adjacent edges $xy$ and $yz$. However, the cut $S$ with $|S| < 5$ separates the cyclic component $C_2$ in a cyclically $5$-connected graph $G$, which is a contradiction.

Thus the vertices $x$ and $z$ are in different components, say $x \in C_1$ and $z \in C_2$. Additionally, we assume that $y \in C_2$; in the case $y \in C_1$ we can proceed analogously. The two neighbours $a_1$ and $a_2$ of $x$ are also in $C_1$ because otherwise we would have two adjacent edges in the minimal cut $S$. Also, the neighbour $a_3$ of $y$ is in the component $C_2$ for the same reason.

We show that $a_4 \in C_2$. Suppose to the contrary that $a_4 \in C_1$. The vertex $a_5$ is in $C_2$ since the adjacent edges $za_4$ and $za_5$ cannot be both in $S$.
The subgraph $G[C_2 - \{y, z\}]$ is separated in $G$ by at most four edges: one edge incident with $a_3$, another one incident with $a_5$ and at most two edges from $S - \{xy, za_4 \}$. Hence $G[C_2 - \{y, z\}]$ has to be acyclic, and since it contains at least two vertices $a_3$ and $a_5$, due to Lemma \ref{lemma:acyclic-pole}, it contains no more vertices and only one edge $a_3a_5$, and also $|S| = 4$. However, this yields a $4$-cycle $a_3a_5zy$ in $\bar{H}$, which is a contradiction. Therefore, $x, a_1, a_2 \in C_1$ and $y, z, a_3, a_4, a_5 \in C_2$ as desired. Finally, if we had $|S| < 4$, then the subgraph $C_1 - x$ would be separated in $G$ by at most four edges: two edges coming from $a_1$ and $a_2$, and at most two edges from $S - \{xy\}$. However, $\dist_H(a_1, a_2) \ge 3$, so $C_1 - x$ has to be cyclic and this is in contradiction with $\zeta(G) = 5$. Therefore $|S| = 4$.
\end{proof}

\section{Main result}
\label{sec:main}

\begin{theorem}
\label{thm:main}
Let $H$ be a cyclic part of a cubic graph $G$ with $\zeta(G) = 5$. If $H$ is not a cycle of length 5, then $H$ can be extended to a cyclically $5$-connected cubic graph by adding three new vertices on a path of length two and by restoring regularity.
\end{theorem}

\begin{proof}
From Lemma \ref{lemma:girth-extension} we know that the graph $H_1 = H(a_1, a_2, a_3, a_4, a_5)$ has girth $5$ for some permutation $a_1a_2a_3a_4a_5$ of the degree $2$ vertices of $H$. If $\zeta(H_1) \ge 5$, we are done, so we assume that $H_1$ contains a cycle-separating cut $S_1$ whose removal leaves components $C_1'$ and $C_2'$. According to Lemma~\ref{lemma:distribution}, $|S_1| = 4$ and without loss of generality $a_1, a_2 \in C_1'$ and $a_3, a_4, a_5 \in C_2'$. Put $C_1 = C_1' - \{x, y, z\}$ and $C_2 = C_2' - \{x, y, z\}$. Denote the three edges of $S$ contained in $H$ by $b_1c_1$, $b_2c_2$ and $b_3c_3$ in such a way that $b_i \in V(C_1)$ and $c_i \in V(C_2)$ for each $i \in \{1,2,3\}$ (see Figure \ref{fig:h1}).

\begin{figure}[h]
\centering
\includegraphics{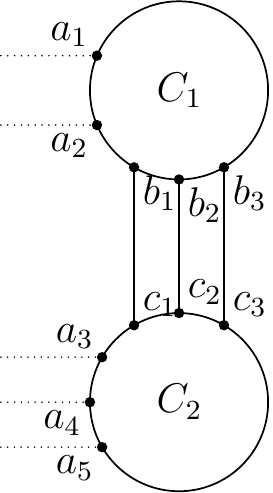}
\caption{The cyclic part $H$ when $\zeta(H_1) < 5$, the edges from $\delta_G(H)$ are dotted}
\label{fig:h1}
\end{figure}

Since the graph $H_1$ has girth at least $5$, we have $\dist_H(a_1, a_2) \ge 3$, $\dist_H(a_4, a_5) \ge 3$, and $\dist_H(a_3, a_i) \ge 2$ for each $i \in \{1, 2, 4, 5\}$. From this we can conclude that $C_1$ contains at least $4$ vertices and $C_2$ contains at least $5$ vertices. Therefore, both $C_1$ and $C_2$ contain a cycle and, moreover, $C_1$ is a fragment and hence, due to Proposition \ref{prop:fragments}, it is $2$-connected.

Suppose that $C_2$ is $2$-connected. According to Lemma \ref{lemma:degrees}, there is at most one edge between the vertices $a_3$, $a_4$, and $a_5$. Therefore, there is a permutation $ijk$ of $\{3, 4, 5\}$ such that $a_ja_i, a_ja_k \notin E(H)$ (and obviously also $a_ja_1, a_ja_2 \notin E(H)$). Moreover, every $a_1$-$a_i$-path and every $a_2$-$a_k$-path contains at least two vertices $b_i$ and $c_i$ for some $i \in \{1, 2, 3\}$. Thus the graph $H_2 = H(a_1, a_i, a_j, a_2, a_k)$ (see Figure \ref{fig:h2a}) has girth at least $5$. Now we show that $\zeta(H_2) \ge 5$. Suppose, to the contrary, that $H_2$ contains a small cycle-separating cut $S_2$. By Lemma \ref{lemma:distribution}, $|S_2| = 4$ and the cut $S_2$ separates $\{x, a_1, a_i\}$ from $\{z, a_2, a_k\}$. However, since both $C_1$ and $C_2$ are $2$-connected, the cut $S_2$ has to contain at least two edges from $C_1$, at least two edges from $C_2$ and the edge $xy$ or $yz$. Therefore $S_2$ contains at least five edges in total; a contradiction.

\begin{figure}[h]
\centering
\begin{subfigure}{0.45\textwidth}
\centering
\includegraphics{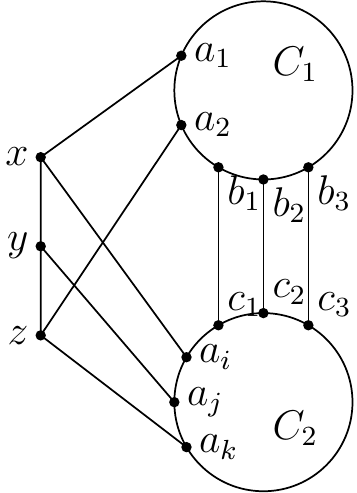}
\caption{$C_2$ is $2$-connected}
\label{fig:h2a}
\end{subfigure}
\begin{subfigure}{0.45\textwidth}
\centering
\includegraphics{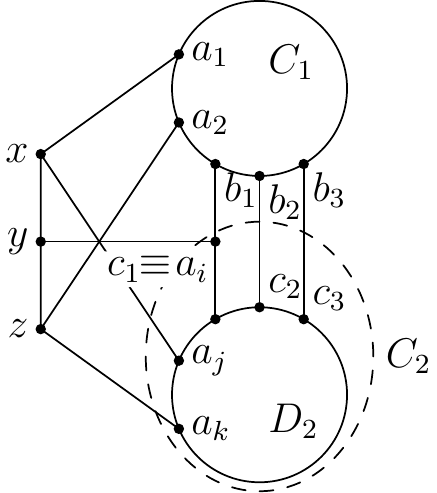}
\caption{$C_2$ contains a bridge}
\label{fig:h2b}
\end{subfigure}
\caption{Completions of the cyclic part $H$ to a cyclically $5$-connected cubic graph}
\label{fig:h2}
\end{figure}

If the cyclic component $C_2$ is not $2$-connected, then by Lemma \ref{lemma:6pole} it contains exactly one bridge such that two edges from $\delta_G(C_2)$ are incident with one end of the bridge and the other end of the bridge lies in a component $D_2$ which is a fragment. Since $a_3$, $a_4$ and $a_5$ are distinct vertices and the same goes also for $c_1$, $c_2$ and $c_3$, we can say that $c_1 \equiv a_i$ for some $i \in \{3, 4, 5\}$. We put $\{3, 4, 5\} - \{i\} = \{j, k\}$ and $H_2 = H(a_1, a_j, a_i, a_k, a_2)$ (see Figure \ref{fig:h2b}).
At first, we show that $g(H_2) \ge 5$. Since $a_i$ is adjacent to $b_1$ and $c_1$, its distance from each of the vertices $a_1$, $a_2$, $a_j$, and $a_k$ is at least $2$. Moreover, we have $\dist_H(a_1, a_j) \ge 3$, because every $a_1$-$a_j$-path contains one of the edges $b_1a_i$, $b_2c_2$ or $b_3c_3$. For the same reason we have also $\dist_H(a_2, a_k) \ge 3$. Hence $g(H_2) = 5$.
Suppose, to the contrary, that $\zeta(H_2) < 5$. According to Lemma \ref{lemma:distribution}, there is a cycle-separating $4$-cut $S_2$ which separates $\{x, a_1, a_j\}$ from $\{ z, a_2, a_k \}$. However, since $C_1$ and $D_2$ are $2$-connected, the cut $S_2$ has to contain at least two edges from $C_1$, at least two edges from $D_2$ and one of the edges $xy$ and $yz$, so in total at least five edges, which contradicts the fact that $|S_2| = 4$. Therefore $\zeta(H_2) = 5$.
\end{proof}

\section{Concluding remarks}

There is only one way how a cyclic part $H$ of a cubic graph with cyclic connectivity $5$ can be completed to a cubic graph by adding fewer than three vertices.
Namely, we can add one new vertex and connect it with three $2$-valent vertices of $H$, and add one new edge between the remaining two $2$-valent vertices.
Assume that $H$ contains three vertices $a_1$, $a_2$ and $a_3$ of degree $2$ such that
all of them have some common neighbour $v$, or there is a $6$-cycle $a_1v_1a_2v_2a_3v_3$ in $H$ (cf. \cite[Lemma 9]{Andersen88}).
Then, in every case, two of the vertices $a_1$, $a_2$ and $a_3$ are connected by an edge or are connected to the newly added vertex $v$, which yields a $3$-cycle or a $4$-cycle, respectively. Thus, the cyclic part $H$ cannot be completed to a cyclically $5$-connected cubic graph by adding only one vertex. Clearly, there are infinitely many cyclic parts satisfying one of the two aforementioned conditions. This stands in contrast to the only exception (the $5$-cycle) for completing $H$ by adding a path of length two.

The following problem remains open.

\medskip
\begin{problem*}
Determine a complete set of conditions under which a cyclic part of a cubic graph with cyclic connectivity $5$ can be completed to a cyclically $5$-connected cubic graph by adding only one additional vertex and restoring $3$-regularity.
\end{problem*}

One may wonder if the subgraph $H$ from Theorem \ref{thm:main} can be completed to a cyclically $5$-connected cubic graph $H'$ by adding a path of length two in such a way that $H'$ is a minor of $G$. As we illustrate in the following example, this is not always true.
 
The graph $G$ from Figure \ref{fig:counterexample} is cyclically $5$-connected and the edge cut $S$ separates $G$ into the component $H$ and a $5$-cycle $v_1v_2v_3v_4v_5$. Let $D$ be the graph from Lemma \ref{lemma:girth-extension}, which contains edges between vertices $a_i$ and $a_j$ if $\dist_H(a_i,a_j) = 2$. In this case, $D$ is the complete graph $K_5$ without the edges $a_1a_4$ and $a_2a_5$. Therefore, $H' = H(a_1, a_5, a_3, a_2, a_4)$ is the only graph of girth $5$ that can be obtained by adding a path of length two to $H$. The graph $H'$ is not a minor of $G$ because the vertices $v_1$ and $v_4$ that should be contracted to $x$ and the vertices $v_2$ and $v_4$ that should be contracted to $z$ lie on the $5$-cycle in an alternating order.

\begin{figure}[h]
    \centering
    \includegraphics{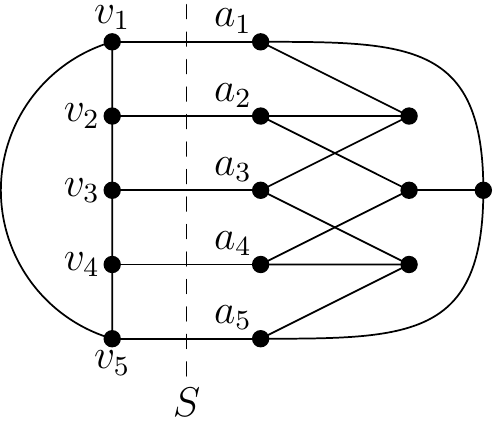}
    \caption{A graph $G$ whose component cannot be completed to a minor of~$G$}
    \label{fig:counterexample}
\end{figure}

\subsection*{Acknowledgements}

The authors acknowledge partial support from the grants VEGA 1/0813/18 and APVV-19-0308.


\begin{thebibliography}{40}
\bibitem{Andersen88} L. Andersen, H. Fleischner, B. Jackson, \textit{Removable edges in cyclically 4-edge-connected cubic graphs}, Graphs Combin. \textbf{4} (1988), 1--21

\bibitem{Goedgebeur} J. Goedgebeur, E. M\'a\v cajová and M. \v Skoviera, \textit{Smallest snarks with oddness 4 and cyclic connectivity 4 have order 44}, Ars Math. Contemp. \textbf{16} (2019), 277--298.

\bibitem{Goedgebeur-oddness-all} J. Goedgebeur, E. M\'a\v cajová and M. \v Skoviera, \textit{The smallest nontrivial snarks of oddness 4}, Discrete Appl. Math. \textbf{277} (2020), 139--162

\bibitem{Jaeger} F.~Jaeger, T. Swart, Problem Session,
    \textit{Combinatorics} 1979, Part II (M. Deza and I. G. Rosenberg,
    eds.), Ann. Discrete Math. \textbf{9} (1980), p.~305.

\bibitem{Kochol04-5f-cc6} M. Kochol, \textit{Reduction of the 5-Flow Conjecture to cyclically 6-edge-connected snarks}, J. Combin. Theory Ser. B \textbf{90}, 139--145 (2004)

\bibitem{Macajova20-bf-cc5} E. M\'a\v cajová, G. Mazzuoccolo, \textit{Reduction of the Berge-Fulkerson conjecture to cyclically 5-edge-connected snarks}, Proc. Amer. Math. Soc. \textbf{148}, 4643--4652 (2020)

\bibitem{Nedela-Atoms} R. Nedela and M. Škoviera, \textit{Atoms of cyclic connectivity in cubic graphs}, Math. Slovaca \textbf{45} (1995), 481–499.

\bibitem{Robertson} N. Robertson, \textit{Minimal cyclic-4-connected graphs}, Trans. Amer. Math. Soc. \textbf{284} (1984), 665–687, doi:10.2307/1999101.

\bibitem{Zhang} C.Q. Zhang, \textit{Integer Flows and Cycle Covers of Graphs}, Marcel Dekker, New York, 1997
\end{thebibliography}
\end{document}